\documentclass[11pt]{article}

\usepackage{amsmath, amsthm, amssymb}
\usepackage{geometry}
\usepackage{hyperref}
\usepackage{mathtools}
\usepackage{microtype}
\usepackage{cite}

\hypersetup{
  colorlinks=true,
  linkcolor=blue,
  citecolor=blue,
  urlcolor=blue
}

\theoremstyle{plain}
\newtheorem{theorem}{Theorem}[section]
\newtheorem{lemma}[theorem]{Lemma}
\newtheorem{corollary}[theorem]{Corollary}

\theoremstyle{definition}

\newcommand{\li}{\text{li}}

\theoremstyle{remark}

\title{\textbf{A Note on Algorithms for Computing $p_n$}}
\author{
Ansh Aggarwal\\
University of Wisconsin–Madison\\
\texttt{ansh@cs.wisc.edu}
}
\date{}

\begin{document}

\maketitle

\begin{abstract}
We analyze algorithms for computing the $n$th prime $p_n$ and establish asymptotic bounds for several approaches. Using existing results on the complexity of evaluating the prime-counting function $\pi(x)$, we show that the binary search approach computes $p_n$ in $O(\sqrt{n} \, (\log n)^4)$ time. Assuming the Riemann Hypothesis and Cramér's conjecture, we construct a tighter interval around $\li^{-1}(n)$, leading to an improved sieve-based algorithm running in $O(\sqrt{n} \, (\log  ^{7/2} n) \, \log \log n)$ time. This improvement, though conditional, suggests that further refinements to prime gap estimates may yield provably faster methods for computing primes.
\end{abstract}

\section{Introduction}

Finding the $n$th prime, denoted $p_n$, given $n$ is an important problem \cite{bachshallit1996}. It is known that an oracle for $\pi(x)$, the prime counting function, allows one to simply run binary search to compute $p_n$ with at most $\log _2 n$ calls to $\pi(x)$ \cite{bachshallit1996}.

 First, given recent developments, we present a simple analysis for the explicit run time for this binary search approach. Notably, we look at Pierre Dusart's bounds for $p_n$ \cite{dusart2010} and look at the fastest known method of computing $\pi(x)$ \cite{hirsch2023}. Here, we conclude that $p_n$ can be found, given $n$, in $O(\sqrt{n} \cdot \log  ^4 n)$ time. 

 Next, we consider sieving to find $p_n$. We show that the fastest sieve, which is a sublinear sieve \cite{sorenson1990}, cannot beat the binary search algorithm asymptotically. Instead, we consider a sieve that finds primes over an interval $[R, L]$ using a segmented sieve. Ultimately, we show that for such a sieve to beat binary search, it must have an interval tighter than some $B$. A look at the literature tells us that an easy bound for $p_n$ that is sufficiently tight does not exist. 

 Finally, we assume the Riemann hypothesis and Cramér's conjecture on prime gaps to show a way to construct an interval tight enough that an interval sieve would asymptotically beat binary search. Thus based on this result, we construct a new algorithm that beats binary search to find $p_n$ given $n$. This algorithm runs in $O(\sqrt{n} \cdot (\log ^{7/2} n) \cdot \log \log n)$ time. 

 In this document, $\log  $ refers taking logarithm with base $e$. We generally denote logarithms in big-$O$ notation with $\log  $. 

\section{Bounds for \texorpdfstring{$p_n$}{pn}}

In \cite{dusart2010}, Pierre Dusart provides bounds for $p_n$, namely, for $n \geq 6$: \begin{align}
    L = n (\log n+ \log (\log n) - 1) < p_n < n(\log n + \log (\log  n)) = R
\end{align} This has the immediate consequence that $p_n$ is always contained in interval of size $n$. \begin{corollary}
    $R - L = n(\log n + \log (\log  n))  -  n (\log n+ \log (\log n) - 1) = n$
\end{corollary} For $n \geq 688383$, Dusart gives a more refined interval of size $\frac{n}{10 \cdot \log n}$. However, since we hope to use this interval with binary search, we show that this improvement does not yield an asymptotic improvement for computing $p_n$. \begin{lemma} 
    Improving the search interval for binary search by a $\log$ factor does not improve the asymptotic run time. 
\end{lemma} \begin{proof}
    We want to show $O(\log n) = O(\log (n/ \log n))$.  Consider the following limit: \begin{align}
        \lim_{n \to \infty}\frac{\log  (n / \log n)}{\log  n} = \lim_{n \to \infty}\frac{\log  n}{\log  n} - \lim_{n \to \infty}\frac{\log  \log n}{\log  n} = 1 - 0 = 1
    \end{align} Thus, under binary search, the bounds for $p_n$ give the same asymptotic result. 
\end{proof}

\section{Using \texorpdfstring{$\pi(x)$}{pi(x)} to Compute \texorpdfstring{$p_n$}{pn}}

The prime counting function, $\pi(x)$, returns the number of prime numbers smaller than or equal to $x$. Formally, this may be written \cite{bachshallit1996} as: \begin{align}
    \pi(x) = \sum_{p \leq x} 1
\end{align} This, however, is not an efficient method to compute $\pi(x)$. Instead, there are two major efficient approaches to computing $\pi(x)$, both of which, in their best forms, run in $\tilde{O}(\sqrt{x})$ time. In \cite{lagarias1987}, Lagarias and Odlyzko present an analytic method to compute $\pi(x)$, assuming the Riemann hypothesis holds true however, the authors note that despite its asymptotic result, it might be impractical to use due to the implied constants, leaving the combinatorial approach as the only real option. 

 In \cite{hirsch2023}, Hirsch, Kessler, and Mendlovic improve to $\tilde{O}(\sqrt{x})$, explicitly $O(\sqrt{x} \cdot  (\log x)^{5/2})$ due to proposition 16 in their paper, with previous notable results attributed Lagarias, Miller, and Oldyzko \cite{lagarias1985} and Rivat and Deléglise \cite{deleglise1996}. 

 In \cite{bachshallit1996}, the authors prove that $p_n$ can be computed quickly given an oracle for $\pi(x)$ using binary search. This is done by finding the smallest $x$ such that $\pi(x) = n$ because all $x' > x$ such that $\pi(x') = n$ are necessarily composite. Let $T_{\pi}$ denote the time complexity of evaluating $\pi(x)$. \begin{theorem}
    Using an oracle for $\pi(x)$, the $n$th prime is computed in $O(T_{\pi} \log n)$ time. 
\end{theorem}  \begin{proof}
    From \textbf{Section 2}, $[L, R]$ defines an interval that contains $p_n$ and it suffices to run binary search over it. The length of the interval is $n \implies$ up to $\log _2 n$ evaluations of $\pi(x)$ are required, thus we have $O(T_{\pi}\log n)$ time. Remark: The change of base from $\log _2$ to $\log  $ is allowed as it does not affect the asymptotic result but allows for cleaner simplifications in the proofs that follow
\end{proof} 

  Thus, using the algorithm from \cite{hirsch2023}, we obtain a more refined result. 

\begin{lemma}
    We can compute $p_n$ in $O(\sqrt{n} \cdot (\log n)^4)$ time.
\end{lemma} \begin{proof} Using the Prime Number Theorem, it is trivial to note that $O(\sqrt{x} \cdot (\log x)^{5/2}) \sim O(\sqrt{n \log n} \cdot (\log (n \log n))^{5/2}) = T_{\pi}$ thus we have: \begin{align}
        O(\sqrt{n \log n} \cdot (\log (n \log n))^{5/2}) \\
        \implies T_{\pi} = O(\sqrt{n} \cdot (\log n)^{3})
    \end{align} Following \textbf{Theorem 3.1}, we conclude that $p_n$ is computed in $O(\sqrt{n} \cdot (\log n)^3 \cdot \log n) \implies O(\sqrt{n} \cdot (\log n)^4)$. 
\end{proof} 

\section{Sublinear Sieves}

In \cite{sorenson1990}, Sorenson presented a sublinear sieve that built on Pritchard's work in \cite{pritchard1981, pritchard1987}. A sieve produces all primes up to some bound $M$; thus, to compute $p_n$, it suffices to sieve up to an upper bound, like $R$ from \textbf{Section 2.1} and counting to $n$th number in the list. With $R$ as the upper bound, the sublinear sieve runs in $O(R / \log \log R)$ time. 
\begin{lemma}
    A sublinear sieve takes $O( (n \log n) / \log \log n)$ time to compute $p_n$.
\end{lemma} \begin{proof}
    From \textbf{Section 2}, $R = n (\log n + \log (\log n))$ gives: \begin{align}
        O \left( \frac{n (\log n + \log (\log n))}{\log (\log n (\log n + \log (\log n))) } \right) \implies O \left( \frac{n \log n }{\log (\log n (\log n + \log (\log n))) }\right)
    \end{align} Let $T$ denote the expression on the right. Note that: \begin{align}
        T \subseteq O \left( \frac{n \log n}{\log \log n}\right)
    \end{align} Counting once sieving is completed takes $O(n)$ time thus, we have $O( (n \log n) / \log \log n)$. 
\end{proof}

\begin{theorem}
    A sublinear sieve is asymptotically worse than the binary search algorithm to compute $p_n$.
\end{theorem} \begin{proof}
    The theorem is true if the following limit converges to $0$: \begin{align}
         \lim_{n \to \infty} \sqrt{n} \cdot (\log n)^{4} \cdot \frac{\log \log n}{n \log n} = \lim_{n \to \infty} \frac{(\log n)^3 (\log \log n)}{\sqrt{n}}
    \end{align} Note that \begin{align}
         \frac{(\log n)^3 (\log \log n)}{\sqrt{n}} \leq \frac{(\log n)^3}{\sqrt{n}}
    \end{align} Thus, if the latter expression converges, so does the original limit, \begin{align}
        \lim_{n \to \infty} \frac{(\log n)^3}{n^{1/2}} = 0
    \end{align} We conclude that this converges to $0$ by a few applications of L'Hopital's rule. 
\end{proof} 

\section{Further Sieving Troubles}

Another way to find $p_n$ is to sieve over a sufficiently small range $[i, j]$ with a counter that is initialized at $\pi(i)$ and incremented for every prime in $[i, j]$. To sieve over such an interval, the best known method involves using segmented sieves that finds all primes up to $\sqrt{j}$ and uses them to mark primes in $[i, j]$ \cite{cpalgos}. Note that such an algorithm runs in $O((j - i + 1)\log \log j + \sqrt{j} \log  \log  \sqrt{j})$ time where $O(\sqrt{j} \log  \log  \sqrt{j})$ comes from needing to find all primes up to $\sqrt{j}$. We assume a stronger variant where the list of primes up to $\sqrt{j}$ is precomputed thus giving $O((j - i + 1) \log  \log j)$ as the time complexity. 

\begin{theorem}
    The size of $[i, j]$ must be less than $\sqrt{n} \cdot \frac{(\log n)^4}{\log \log n}$ (up to some constant) in order to beat the binary search algorithm asymptotically when attempting to find $p_n$.  
\end{theorem} \begin{proof}
    First, we let $s = (j - i +1)$ and in order to isolate the dependence on $s$ instead of the upper bound $j$ alone, we let $j \sim n \log n$ giving the time complexity of the interval sieve as: \begin{align}
        O(s \log \log (n \log n)) = O(s \log \log n)
    \end{align}
    From the previous sections, we note that the best binary search algorithm runs in $O(\sqrt{n} \cdot (\log n)^4)$ thus, for the sieve to beat the algorithm, the following inequality must hold for some $s$ and some constant $c$: \begin{align}
    c\cdot \frac{s \log \log n}{\sqrt{n} \cdot (\log n)^4} <1 \implies s < \frac{1}{c} \cdot \frac{\sqrt{n} \cdot (\log n)^4}{\log \log n} = B
\end{align} If $|[i, j]| = s < B$, then a sieve over $[i, j]$ would be asymptotically faster than the binary search algorithm. 
\end{proof} 

 Dusart's bounds show that $p_n$ is found in an interval of size $O(n / \log n)$ as noted in \textbf{Section 2} \cite{dusart2010} and while, further work, notably by Christian Axler \cite{axler2018}, does improve bounds for $p_n$, asymptotically, they yield $O(n)$. As such, with current state of bounds for $p_n$ a segmented sieve over an interval cannot beat the binary search algorithm without further work.

\section{Using Schoenfeld's Bound}

A consequence of the Riemann hypothesis, as given by Lowell Schoenfeld \cite{schoenfeld1976}, is the following bound: \begin{align}
    |\pi(x)-\li(x)| \leq \frac{1}{8 \pi} \cdot\sqrt{x} \cdot \log x
\end{align} Where $\li(x)$ is the logarithmic integral function. A consequence of this result given in the following corollary: 

\begin{corollary} For $k = \frac{1}{8 \pi}$, \begin{align}
    \pi(x) \in I = [\li (x) - k\sqrt{x} \log x, \li(x) + k\sqrt{x} \log x]
\end{align}
\end{corollary}

 For an $\alpha$ such that $\li(\alpha) = n$, we have that $\pi(\alpha) = n' \in I$. Consider the two possibilities, if $n' > n$, then $p_n < \alpha$ and if $n' < n$, then $p_n > \alpha$. 

 Unlike $\pi(x)$, $\li(x)$ is easy to approximately compute using the following expansion: \begin{align}
    \li(x) \sim  \frac{x}{\log x} \sum_{k = 0}^{\infty} \frac{k!}{(\log x)^k}
\end{align} This means for any evaluation of $\li(x)$ with precision $\varepsilon$, the time complexity is $O(\log (1 / \varepsilon)) \implies O(1)$, constant time. As such, using $\li(x)$ almost as a `proxy' for $\pi(x)$ by evaluating $\li(x)$ $\log _2 n$ times to get $\alpha$ such that $\li(\alpha) = n$, done entirely in $O(\log _2 n) \implies O(\log n)$ time. The obstacle that remains is going from $\alpha$ to $p_n$ however, here we may consider Cramér's conjecture \cite{granville1995}: \begin{align}
    p_{n + 1} - p_n \in O( (\log  p_n)^2)
\end{align} A natural extension of this conjecture is given for larger gaps, \begin{align}
    p_{n + \Delta} - p_n \in O( \Delta (\log p_{n + \Delta })^2)
\end{align} Note that while $\alpha$ is not guaranteed to be a prime, we may still use it as though it were a prime and $p_n$ together to bound $|p_n - \alpha|$ asymptotically. \begin{theorem}
    The gap $|p_n - \alpha|$ is $O(\sqrt{n} \cdot (\log ^{7/2} n))$ where $\pi(p_n) = n$ and $\li(\alpha) = n$. 
\end{theorem} \begin{proof}
    From Schoenfeld's bound, \begin{align}
        |\pi(\alpha)-\li(\alpha)| \leq \frac{1}{8 \pi} \cdot\sqrt{\alpha} \cdot \log \alpha
    \end{align} Additionally, for some $\Delta$ using Cramér's conjecture, we have, \begin{align}
        |\alpha - p_n| \in O(\Delta((\log \max(\alpha, p_n))^2)
    \end{align} Note that the bound gives a precise value for $\Delta$, explicitly $\Delta = \frac{1}{8 \pi} \cdot\sqrt{\alpha} \cdot \log \alpha$ and, without loss of generality, let $\max(\alpha, p_n) = \alpha$, then we have: \begin{align}
        |\alpha - p_n| \in O(\sqrt{\alpha} \cdot \log  ^3 \alpha)
    \end{align} Now, we use the Prime Number Theorem $\alpha \sim p_n \sim n \log n$, giving: \begin{align}
        |\alpha - p_n| \in O ( \sqrt{n \log n} \cdot \log  ^3(n \log n)) \implies O(\sqrt{n} \cdot (\log  ^{7/2} n))
    \end{align}
    We have our result, $O(\sqrt{n} \cdot (\log  ^{7/2} n))$. 
\end{proof} 

 The immediate consequence of this theorem is that attempting to sieve over the interval around $\alpha$ that contains $p_n$ necessarily beats the binary search algorithm, as given by the theorem in \textbf{Section 5}.

\section{An Improved Algorithm}

Following the discussion in the previous section, a consolidated algorithm to find $p_n$ is as follows: \begin{enumerate}
    \item Evaluate $\li(x)$ $\log _2 n$ many times over the interval given by Dusart's bounds to find $\alpha$ such that $\li(\alpha) = n$. This is done in $O(\log n)$ time. 
    
    \item Define an interval $[L, R] = [\alpha - O(\sqrt{n} \cdot (\log ^{7/2} n)), \alpha + O(\sqrt{n} \cdot (\log ^{7/2} n))]$. Sieving over this interval takes $O(\sqrt{n} \cdot (\log ^{7/2} n)) \cdot \log \log n)$ time. 

    \item With the list of primes in the interval, evaluate $\pi(p')$ once for any $p'$ in the list and simply count up / down to the $n$th prime. This is done in ${O}(\sqrt{p'} \cdot (\log p')^{5/2}) \implies {O}(\sqrt{n \log n} \cdot (\log (n \log n))^{5/2}) \implies O(\sqrt{n} \cdot \log ^3 n)$ time. 
\end{enumerate} Overall, this algorithm runs in: \begin{align}
    O(\log n + \sqrt{n} \cdot (\log ^{7/2} n)) \cdot \log \log n + \sqrt{n} \cdot \log ^3 n) \implies O(\sqrt{n} \cdot (\log ^{7/2} n) \cdot \log \log n)    
\end{align} This is faster than binary search with multiple $\pi(x)$ evaluations, which runs in $O(\sqrt{n} \cdot \log  ^4 n)$ time. 

\section{Conclusion}

In this note, we analyzed multiple algorithms for computing the $n$th prime $p_n$ and established concrete asymptotic bounds for their performance. Using existing results on the complexity of evaluating the prime-counting function $\pi(x)$, we showed that the binary search approach computes $p_n$ in $O(\sqrt{n} \cdot (\log n)^4)$ time. We further demonstrated that even the fastest known sublinear or segmented sieves cannot asymptotically outperform this method with current bounds for $p_n$. 

 Assuming the Riemann Hypothesis and Cramér’s conjecture, we constructed a tighter interval around the approximation $\alpha$ satisfying $\li(\alpha) = n$. This led to an improved sieve-based algorithm running in $O(\sqrt{n} \cdot (\log  ^{7/2} n) \cdot \log \log n)$ time, which asymptotically beats the binary search approach. 

 While the improvement is conditional and primarily theoretical, it suggests that future refinements in prime gap estimates or unconditional bounds on $p_n$ could lead to provably faster methods for computing primes.

\section*{Acknowledgements}
We would like to thank Prof.\ Eric Bach and Prof.\ Jonathan Sorenson for insightful discussions and feedback during the preparation of this work. We also thank Prof.\ Rahul Chatterjee for his support and guidance. Finally, we'd like to thank Prof. Jeffrey Shallit for providing comments on the presentation of this work.

\end{document}